\documentclass[a4paper]{article}
\usepackage{amsmath}
\usepackage{mathrsfs}
\usepackage{amsfonts}
\usepackage{graphicx}
\usepackage{amsmath}
\usepackage{amssymb}
\usepackage[all]{xypic}
\usepackage[all]{xy}
\usepackage[toc,page,title,titletoc,header]{appendix}
\usepackage{xcolor}
\usepackage{hyperref}
\usepackage[square,numbers]{natbib}
\usepackage{fancyhdr}
\usepackage{amsthm}
\usepackage[a4paper, margin=3cm]{geometry}

\newcommand{\tild}{{\tilde d}}
\newcommand{\tils}{{\tilde s}}

\newcommand{\tilD}{{\widetilde D}}

\newcommand{\tilGamma}{{\widetilde{\Gamma}}}
\newcommand{\tilK}{{\widetilde K}}

\newcommand{\tilP}{{\widetilde{P}}}
\newcommand{\tilS}{{\widetilde S}}

\newcommand{\tilX}{{\widetilde{X}}}

\renewcommand{\P}{{\mathbb P}}

\newcommand{\Q}{{\mathbb Q}}
\newcommand{\Z}{{\mathbb Z}}
\newcommand{\N}{{\mathbb N}}

\newcommand{\SL}{\mathrm{SL}}

\newcommand{\Char}{\mathrm{Char}}

\newcommand{\Spec}{\mathrm{Spec}}

\newcommand{\NN}{\mathcal{N}}
\newcommand{\OO}{\mathcal{O}}
\newcommand{\gerp}{{\mathfrak{p}}}
\newcommand{\gerP}{{\mathfrak{P}}}
\newcommand{\height}{\mathrm{h}}

\newtheorem{theorem}{THEOREM}[section]

\newtheorem{corollary}[theorem]{COROLLARY}

\newtheorem{lemma}[theorem]{LEMMA}

\newtheorem{proposition}[theorem]{PROPOSITION}
\newtheorem{remark}[theorem]{REMARK}

\begin{document}

\title
{Integral points on the modular curves $X_0(p)$}

\author{Yulin Cai}

\newcommand{\address}{{
	\bigskip
	\footnotesize
	Y.~Cai,
	\textsc{Institut de Math\'ematiques de Bordeaux, Universit\'e de Bordeaux 
	351, cours de la Lib\'eration 33405 Talence Cedex, France}\par\nopagebreak

    \textit{E-mail address}: \texttt{yulin.cai1990@gmail.com}
    }}

\maketitle

\begin{abstract}
	
	In this paper, we give an explicit bound for the height of integral points on $X_0(p)$ by using a very explicit version of the Chevalley-Weil principle. We improve the bound given by Sha in \cite{sha2014bounding1}. 
\end{abstract}

{\footnotesize
	
	\tableofcontents
	
}

\section{Introduction}

Let~$X$ be a smooth, connected projective curve defined over a number field~$K$, and let ${x\in K(X)}$ be a non-constant rational function on~$X$. If~$R$ is a subring of~$K$, we denote by $X(R,x)$ the set of $R$-integral $K$-rational points of~$X$ with respect to $x$:
$$
X(R,x)=\{P\in X(K):x(P)\in R\}. 
$$
In particular, if~$S$ is a finite set of places of~$K$ (including all the infinite places), we consider the set of \textit{$S$-integral points} $X(\OO_S,x)$, where ${\OO_S=\OO_{S,K}}$ is the ring of $S$-integers in~$K$.

According to the classical theorem of Siegel~\cite{siegle1929uber} (see also \cite[Part~D]{hindry2000diophantine} for a modern exposition), the set $X(\OO_S,x)$ is finite if at least one of the following conditions is satisfied: 
\begin{align}
&g(X)\geq 1;\\
\label{ethreepoles}
&\text{$x$ admits at least~$3$ poles in $X(\bar\Q)$}. 
\end{align}
The theorem of Faltings~\cite{faltings1983end} (see also \cite[Part~E]{hindry2000diophantine}) asserts that $X(K)$ is finite if ${g(X)\geq 2}$. Unfortunately, all known proofs of the theorems of Siegel and Faltings are non-effective, which means that they do not imply any explicit expression bounding the heights of integral or rational points. 

Starting from the ground-breaking work of A.~Baker in 1960th, effective proofs of Siegel's theorem were discovered, by Baker and others, for many pairs $(X,x)$, see \cite{bilu1995effctive,bilu2002baker} and the references therein.

One interesting case is  when ${X=X_\Gamma}$ is the modular curve corresponding to a subgroup~$\Gamma$ of ${\Gamma(1)=\SL_2(\Z)}$, and ${x=j}$ is the rational function defined by the $j$-invariant.

Bilu~\cite[Section~5]{bilu1995effctive} (see also \cite[Section 4]{bilu2002baker}) made the following observation.

\begin{proposition}
	\label{prthree}
	Let $\Gamma$ be a congruence subgroup of $\SL_2(\Z)$ of level $N$ having at least $3$ cusps. Let $K$ be a number field such that $X_\Gamma$ admits a geometrically irreducible model over $K$ and such that ${j\in K(X_\Gamma)}$. Let~$S$ be a finite set of places of $K$ containing all the infinite places. Then there exists an effective constant ${c=c(N,K,S)}$ such that for any ${P\in X_\Gamma(\OO_S,j)}$ we have ${\height(j(P))\leq c}$. 
\end{proposition}

(Here ${\height(\cdot)}$ is the standard absolute logarithmic  height defined on the set $\bar\Q$ of algebraic numbers.) 

In other words, if condition~\eqref{ethreepoles} is satisfied for the couple $(X_\Gamma,j)$, then Siegel's theorem is effective for this couple. 

Sha~\cite{sha2014bounding} made the bound in Proposition~\ref{prthree} totally explicit. The full statement of his result is quite involved, and we reproduce it only  in Section~\ref{sproof}, see Theorem~\ref{thsha} therein. Here we only notice that Sha's bound is of the shape
${c(K,S)^{N\log N}}$, where $c(K,S)$ is an effective constant depending only on~$K$ and~$S$. Roughly speaking, we have here exponential type dependence in~$N$.

Proposition~\ref{prthree} applies in many important cases: see \cite{bilu2002baker,bilu2011effective} for further details.  In particular, it applies to the  modular curve $X_0(N)$ of composite level~$N$. However, it does not directly apply to the curve $X_0(p)$ of prime level~$p$, because it has only~$2$ cusps.

Nevertheless, using a covering argument, Bilu~\cite[Theorem~10]{bilu2002baker} proved that Siegel's theorem is effective for $X_0(p)$ as well. Note that  the curve $X_0(N)$ has a standard geometrically irreducible model over~$\Q$. 

\begin{theorem}[Bilu]
	\label{thnonexp}
	Let~$p$ be a prime number distinct from $2,3,5,7,13$. Let~$K$ be a number field and~$S$ be  a finite set of places of~$K$ containing all the infinite places. Then there exists an effective constant ${c=c(p,K,S)}$ such that for any ${P\in X_0(p)(\OO_S,j)}$ we have ${\height(j(P))\le c}$. 
	\label{bilu X_0(p)}
\end{theorem}

The main tool is the classical \textit{Chevalley-Weil Principle}, used in the following form. 

\begin{proposition}[Chevalley-Weil Principle]
	Let ${\tilX\stackrel\pi\to X}$ be a non-constant \'etale morphism of projective algebraic curves defined over a number field~$K$. Then there exists a finite set~$T$ of places of~$K$ such that the following holds. 
	Let ${P\in X(\bar K)}$ and let ${\tilP \in \tilX(\bar K)}$ be such that ${\pi(\tilP)=P}$. Let~$v$ be a finite place of the field $K(P)$ ramified in $K(\tilP)$. Then~$v$ extends a place from~$T$.   
\end{proposition}

Bilu found a subgroup~$\tilGamma$ of $\Gamma_0(p)$ such that the natural morphism ${X_\tilGamma\to X_0(p)}$ is \'{e}tale and~$X_\tilGamma$ has at least three cusps, see Proposition~\ref{pgamprime} for the details.  The Chevalley-Weil principle now allows one to reduce the problem from $X_0(p)$ to $X_\tilGamma$, where Proposition~\ref{prthree} applies. 

In \cite{sha2014bounding1} Sha gave an explicit version of Theorem \ref{thnonexp}. We again do not reproduce here Sha's full statement, which is very involved, and only focus on the dependence on the level~$p$. One can expect here exponential type dependence in $p$, but Sha obtains  an upper bound of the form $c(K,S)^{\exp(p^6\log p)}$, doubly exponential in $p$. 

Sha's bound is so big because he uses a quantitative version of the Chevalley-Weil Principle from~\cite{bilu2013quantitative}, which provides  extremely high upper bounds for the quantities involved.

In this paper, we will use Igusa's theorem, see \cite[Section 8.6]{diamond2005first}, combined with Proposition \ref{model} and Lemma \ref{uramified}, to have a very explicit version of the Chevalley-Weil principle in the special case we need. Then 
we manage to improve the result of Sha by using it instead of the general quantitative Chevelley-Weil principle from \cite{sha2014bounding1}. For a finite place~$v$ of a number field~$K$ we denote by $\NN_{K/\Q} (v)$ the absolute norm of the prime ideal corresponding to~$v$. We will prove the following theorem.

\begin{theorem}
    Keep the notations of Theorem \ref{bilu X_0(p)}. Then for $ P \in X_0(p)(\OO_S,j)$, we have 
	$$\height(j(P)) \leq  e^{9s^2p^4\log p}C(K,S)^{p^2},$$
	where $C(K,S)$ can be effectively determined in terms of K and S. More explicitly, $C(K,S)$ can be chose as 
	$$C(K,S)=2^{31s}d^{9s}s^{2s}\ell^d|D|(\log{(|D|+1)})^{d}\prod\limits_{\substack{v \in S \\ v  \nmid \infty}}\log\mathcal{N}_{K/\mathbb{Q}}(v),$$
	where $d=[K:Q]$, $D$ is the absolute discriminant of $K$, $s = |S|$, and $\ell$ is the maximal prime such that there exists $v\in S$ with $v|\ell$. 
	\label{main}
\end{theorem}

\section{Lemmas}

For the convenience of the readers, we state a result from \cite{liu1999models}.

\begin{proposition}[\cite{liu1999models} Corollary 4.10]
	Let $K$ be a discrete valuation field with ring of integers $\mathcal{O}_K$, and $f: X \rightarrow Y$ be a finite morphism of smooth, connected projective curves over $K$.
	Assume that $g(Y) \geq 1$, and that $X$ admits a smooth projective model $\mathcal{X}$. Then $Y$ admits a smooth projective model $\mathcal{Y}$, and $f$ extends to a finite morphism $\mathcal{X} \rightarrow \mathcal{Y}$.
   \label{model}
\end{proposition}

\begin{lemma}
Let $f: X\rightarrow Y$ be a finite, \'{e}tale morphism of smooth, connected projective curves over a number field $K$ with $g(Y) \geq 1$, and let $\mathfrak{p} \subset \mathcal{O}_K$ be a non-zero prime with residue field $k(\mathfrak{p})$. Suppose that 
\begin{enumerate}
	\item[(1)] $X$ admits a smooth projective model at $\mathfrak{p}$;
	\item[(2)] $[K(X):K(Y)] < \Char(k(\mathfrak{p}))$ or $K(X)/K(Y)$ is Galois of degree prime to $\Char(k(\mathfrak{p}))$.
\end{enumerate}
 Then for every point $P \in Y(K)$ and $Q\in f^{-1}(P)$, we have that $\mathfrak{p}$ is unramified in the residue field $K(Q)$ of $Q$. 
\label{uramified}
\end{lemma}
\begin{proof}
	Suppose that $\mathcal{X}$ is the smooth model of $X$ over $\mbox{Spec}(\mathcal{O}_{K,\mathfrak{p}})$. Since $f$ is finite, and $g(Y) \geq 1$, then by Proposition \ref{model}, $Y$ admits a smooth model $\mathcal{Y}$ and $f$ is extended to a finite morphism $\mathcal{X} \rightarrow \mathcal{Y}$. We still denote the extended morphism by $f$. 	
	
	We endow the closure $\overline{\{P\}}$ of $\{P\}$ in $\mathcal{Y}$ with structure of reduced closed subscheme. It is a section of $\mathcal{Y}$ over  $\mbox{Spec}(\mathcal{O}_{K,\mathfrak{p}})$, that is because  $P \in Y(K)$, and $\overline{\{P\}}$ is finite, birational over $\mbox{Spec}(\mathcal{O}_{K,\mathfrak{p}})$. Consider $\mathcal{X}\times_{\mathcal{Y}}\overline{\{P\}}$. It is finite over $\overline{\{P\}} \simeq \mbox{Spec}(\mathcal{O}_{K,\mathfrak{p}})$, hence affine, denoted by $\mbox{Spec}(A)$. Its underlying space is $f^{-1}(\overline{\{P\}})$. If $\mathcal{X} \rightarrow {\mathcal{Y}}$ is \'{e}tale, then after the base change $\overline{\{P\}} \rightarrow \mathcal{Y}$, $\mbox{Spec}(A) \rightarrow \mbox{Spec}(\mathcal{O}_{K,\mathfrak{p}})$ is also \'{e}tale. Since $\mathcal{O}_{K,\mathfrak{p}}$ is regular, so $A$ is regular too. Suppose that $A = \bigoplus\limits_{i=1}^{m}A_i$ such that $A_i$ is connected for each $i$. The fact that $A$ is regular and finite over $\mathcal{O}_{K,\mathfrak{p}}$ implies that $A_i$ is normal and finite over $\mathcal{O}_{K,\mathfrak{p}}$ for each $i$. In particular, the affine ring corresponding to $\overline{\{Q\}}$ is the integral closure of $\mathcal{O}_{K,\mathfrak{p}}$ in $K(Q)$. Any closed point $x$ on $\overline{\{Q\}}$ is also a closed point on $\mbox{Spec}(A)$. We know that $\overline{\{Q\}}$ and $\Spec(A)$ have the same local rings at $x$, so $\overline{\{Q\}} \rightarrow \overline{\{P\}}$ is \'{e}tale at $x$. Hence $\mathfrak{p}$ is unramified in $K(Q)$.
	
	It remains to show that  $\mathcal{X} \rightarrow {\mathcal{Y}}$ is \'{e}tale. Let $Z$ be the set of points in $\mathcal{X}$ at which $f$ is not \'{e}tale, then $Z$ is closed in $\mathcal{X}$. If $Z \neq \emptyset$, since $Z \neq \mathcal{X}$, by Zariski-Nagata purity theorem in \cite[Th\'eor\`em de puret\'e 3.1]{grothendieck2002rev}, it is purely of codimension $1$.  Any irreducible component $W$ of $Z$ is vertical, because $X \rightarrow Y$ is \'etale.  Let $\eta$ be the generic point of $W$, then $\xi = f(\eta) \in \mathcal{Y}$ is also a generic point in $\mathcal{X}_s$ from the fact that $f$ is dominant and finite, where $\mathcal{X}_s$ is the special fiber of $\mathcal{X}$. Consider $f^\#_\eta:\mathcal{O}_{\mathcal{Y},\xi} \rightarrow \mathcal{O}_{\mathcal{X},\eta}$. We claim that the maximal ideals of $\mathcal{O}_{\mathcal{Y},\xi}$ and $\mathcal{O}_{\mathcal{X},\eta}$ are $\mathfrak{p}\mathcal{O}_{\mathcal{Y},\xi}$ and $\mathfrak{p}\mathcal{O}_{\mathcal{X},\eta}$ respectively. Indeed, we have that $\mathcal{O}_{\mathcal{X}_s,\eta} = \mathcal{O}_{\mathcal{X},\eta}/\mathfrak{p}\mathcal{O}_{\mathcal{X},\eta}$, and the special fiber $\mathcal{X}_s$ is smooth, so  $\mathcal{O}_{\mathcal{X}_s,\eta}$ is integral with only one prime ideal. Hence $\mathcal{O}_{\mathcal{X}_s,\eta}$ is a field, and $\mathfrak{p}\mathcal{O}_{\mathcal{X},\eta}$ is the maximal ideal of $\mathcal{O}_{\mathcal{X},\eta}$. It is similar for $\mathcal{O}_{\mathcal{Y},\xi}$. On the other hand, $[k(\eta):k(\xi)] \leq [K(X):K(Y)]$ and $[k(\eta):k(\xi)] | [K(X):K(Y)]$ if $X \rightarrow Y$ is Galois. By the assumption (2), the residue degree $[k(\eta): k(\xi)] < \Char(k(\mathfrak{p}))$ or $[k(\eta): k(\xi)] $ prime to $ \Char(k(\mathfrak{p}))$, so $k(\eta)/k(\xi)$ is separable. Hence $\mathcal{O}_{\mathcal{Y},\xi} \rightarrow \mathcal{O}_{\mathcal{X},\eta}$ is unramified. It is also flat since it is injective and $\mathcal{O}_{\mathcal{Y},\xi}$ is a Dedekind domain, hence also \'{e}tale. Contradiction. 
\end{proof}

We will combine the lemma above with the following lemma in the sequel.

\begin{lemma}
	\label{lded}
	Let $L/K$ be a finite extension of number fields and ~$T$ be a finite set of prime numbers such that every ramified place is above a prime from~$T$. Then 
	$$
	\bigl|\NN_{K/\Q}(D_{L/K})\bigr|\le \left(\prod_{p\in T}p\right)^{[L:\Q]^2},
	$$
	where $D_{L/K}$ is the discriminant of $L$ over $K$.
\end{lemma}

\begin{proof}
	The ``Dedekind Discriminant Formula'' \cite[Theorem B.2.12]{bombieri2007heights} implies that
	$$
	\nu_\gerp(D_{L/K}) = \sum_{\gerP\mid \gerp}(e_{\gerP/\gerp}(1+\delta_{\gerP})-1)f_{\gerP/\gerp} \leq [K:\Q]\sum_{\gerP\mid \gerp}e^2_{\gerP/\gerp}f_{\gerP/\gerp}\leq [L:K][L:\Q],
	$$
	where~$\gerp$ is a prime of~$K$ ramified in~$L$, the sum is over the primes of~$L$ above~$\gerp$, and $0 \leq \delta_\gerP \leq v_\gerp(e_{\gerP/\gerp}) < [K:\Q]e_{\gerP/\gerp}$. For every such~$\gerp$ we have ${\bigl|\NN_{K/\Q}\gerp\bigr| = p^{f_{\gerp/p}}}$, where~$p$ is the prime number below~$\gerp$. Hence 
	$$
	\bigl|\NN_{K/\Q}(D_{L/K})\bigr|\le \left(\prod_{p\in T}p^{\sum_{\gerp\mid p}f_{\gerp/p}}\right)^{[L:K][L:\Q]} \le \left(\prod_{p\in T}p^{[K:\Q]}\right)^{[L:K][L:\Q]}= \left(\prod_{p\in T}p\right)^{[L:\Q]^2}.
	$$
\end{proof}

\section{Proof of  theorem \ref{main} \label{sproof}}

\subsection{An \'{e}tale covering \label{etale covering}}

\label{smorph}
Let~$\tilGamma$ be the subgroup of $\Gamma_0(p)$ defined as follows: set $A = \{a\in \mathbb{F}_p^*: a^{12}=1\}$, and 
\begin{equation}
\label{egtilde}
\tilGamma=\left\{\begin{bmatrix}a&b\\c&d\end{bmatrix}\in \Gamma_0(p): a \bmod p \in A\right\}. 
\end{equation}
It is not hard to see that  the curve $X_{\tilGamma}$ and the natural morphisms $X_1(p) \rightarrow {X_{\tilGamma}\stackrel\pi\to X_0(p)}$ are defined over~$\Q$.   
\begin{proposition}
	\label{pgamprime}
	\begin{enumerate}
		\item
			We have $\deg \pi \leq \dfrac{p-1}{2}$.
		
		\item
		When ${p\notin\{2,3,5,7,13\}}$, the curve $X_{\tilGamma}$ has at least~$3$ cusps.
	
		\item
		\label{ietale}
		The  morphism~$\pi$   is \'{e}tale.  
	\end{enumerate}
\end{proposition}

\begin{proof}
	Set $\overline{\tilGamma}$ the image of $\tilGamma$ in $\SL_2(\mathbb{F}_p)$, then we have
	$$\mbox{deg}\pi = [\Gamma_0(p):\tilGamma] = [\mbox{ST}_2(\mathbb{F}_p):\overline{\tilGamma}] = p(p-1)/(p|A|) \leq \dfrac{p-1}{2}.$$
	
	The second assertion is proved in \cite[page~84]{bilu2002baker}.
	
	About the third assertion, it is only proved in~\cite{bilu2002baker} that~$\pi$ is \'{e}tale outside the cusps. In fact, $\pi$ is \'{e}tale at the cusps as well. Indeed, the $j$-map ${X(p)\stackrel j\to \P^1}$ has ramification index~$p$ at every cusp. Hence~$1$ and~$p$ are the only possible ramification indices for~$\pi$. Since ${\deg \pi \le (p-1)/2<p}$, the ramification indices at the cusps are all~$1$.
\end{proof}

\begin{corollary}
	\label{cchw}
	Let $K$ be a number field, $P \in X_0(p)(K)$ and $\tilP \in \pi^{-1}(P)$. Then 
	\begin{align}
	\label{ereldeg}
	[\tilK:K] \leq \dfrac{p-1}{2},\\
	\label{ereldis}
	\bigl|\NN_{K/\Q}(D_{\tilK/K})\bigr| \leq p^{d^2(p-1)^3/8},
	\end{align}
	where $\tilK = K(\tilP)$, the residue field of $\tilP$, and $d = [K:\Q]$.
\end{corollary}

\begin{proof}
	It follows from Proposition \ref{pgamprime} and the formula
	$$\mbox{deg} \pi = \sum\limits_{Q\in \pi^{-1}(P)}[K(Q):K]$$
	that 
	$$
	[\tilK:K]\le \deg \pi \le (p-1)/2.
	$$
	We know that the modular curve $X_1(p)$ has good reductions outside $p$ by Igusa's Theorem, see \cite[Section 8.6]{diamond2005first}. Now by Proposition \ref{model}, $X_\tilGamma$ also admits good reduction outside $p$. Combining this with  Proposition \ref{pgamprime},  Lemma \ref{uramified} and the fact that $[K(X_\tilGamma):K(X_0(p))] = \deg\pi \leq \frac{p-1}{2}$, we apply Lemma~\ref{lded}  with $T=\{q: q \leq (p-1)/2, \textrm{$q$ is prime} \} \cup\{p\}$, we obtain~\eqref{ereldis}.  
\end{proof}

\subsection{Calculations}

For a number field $K$, and a finite subset $S \subseteq M_K$ containing all infinite places, we put $d=[K:\mathbb{Q}]$ and $s = |S|$. Let $\mathcal{O}_K$ be the ring of integers of $K$. We define the following quantity
$$\Delta_0(N): = \sqrt{N^{dN}|D|^{\varphi(N)}}(\mbox{log}(N^{dN}|D|^{\varphi(N)}))^{d\varphi(N)}\times \left(\prod\limits_{\substack{v \in S\\  v\nmid \infty}}\mbox{log}\mathcal{N}_{K/\mathbb{Q}}(v)\right)^{\varphi(N)}$$
as a function of $N \in \N^+$, where $D$ is the absolute discriminant of $K$, $\varphi(N)$ is the Euler's totient function, and the norm $\mathcal{N}_{K/\mathbb{Q}}(v)$ of a place $v$, by definition, is equal to $|\mathcal{O}_K/\mathfrak{p}_v|$ when $v$ is finite and $\mathfrak{p}_v$ is its corresponding prime ideal, and is set to be $1$ if $v$ is infinite.

With these notations above, the main tool is following theorem proved by M.Sha in \cite{sha2014bounding}. The form of this theorem with explicit constant will be used in our proof of Theorem \ref{main}.

\begin{theorem}[\cite{sha2014bounding} Theorem 1.2]
	\label{thsha}
	Let $\Gamma$ be a congruence subgroup of level $N$ and $X_\Gamma$ be the corresponding modular curve over a number field $K$ with $d=[K:\mathbb{Q}]$, and $S \subseteq M_K$ be a finite set containing all archimedean places. If $v_\infty(\Gamma) \geq 3$, then for any $ P \in X_\Gamma(\mathcal{O}_S,j)$, the following holds.
	\begin{itemize}
		\item[(1)] If $N$ is not a power of any prime, we have 
		$$\height(j(P))\leq (CdsN^2)^{2sN}(\log(dN))^{3sN}\ell^{dN}\Delta_0(N),$$
		where $C$ is an absolute effective constant, and $\ell$ is the maximal prime such that there exists $v\in S$ with $v|\ell$.
		\item[(2)] If $N$ is a power of some prime, we have 
		$$\height(j(P))\leq (CdsM^2)^{2sM}(\log(dM))^{3sM}\ell^{dM}\Delta_0(M),$$
		where $C$ is an absolute effective constant, $\ell$ is the maximal prime such that there exists $v\in S$ with $v|\ell$, and $M$ is defined as following: $M=3N$ if $N$ is a power of $2$, and $M=2N$ otherwise.
	\end{itemize}
	
	\label{bound}
\end{theorem} 

\begin{remark}
	By following the calculation carefully, see \cite{cai2019explicit}, one can show that $2^{15}$ is a suitable value for the constant $C$. 
\end{remark}
\

We begin to prove Theorem \ref{main}. Recall the congruence subgroup $\tilGamma \subset \Gamma_0(p)$ defined in subsection \ref{etale covering}, i.e.  
$$\tilGamma=\{\left[
{\begin{array}{cc}
	a&b\\
	c&d
	\end{array}}\right]\in \Gamma_0(p):a\ \mbox{mod} \ p \in A\},$$ 
where $A = \{a\in \mathbb{F}_p^*: a^{12}=1\}$, and the natural map $\pi: X_{\tilGamma} \rightarrow X_0(p)$. For any $ P \in X_0(p)(\mathcal{O}_S,j)$, there exist a finite extension $\tilK$ of $K$ and $\tilP \in X_{\tilGamma}(\tilK)$ such that $\pi(\tilP)=P$. A non-constant morphism between irreducible projective curves is always surjective, so $\pi(X_{\tilGamma})=X_0(p)$. Obviously, $\height(j(\tilP)) = \height(j(P))$, so it's sufficient to bound $\height(j(\tilP))$. Hence we consider the points in $X_{\tilGamma}(\mathcal{O}_{\tilS},j)$, where 
$$\tilS=\{v\in M_{\tilK}: v|w\ \mbox{for some}\ w\in S\}.$$
By Proposition \ref{pgamprime}, we know that $X_{\tilGamma}$ has at least three cusps.

To apply Theorem \ref{bound} for $\height(j(\tilP))$, we should bound some invariants of $\tilK$ and $\tilS$. 
We fix some notations before proceeding with the proof, we set 
$$\tilde{\Delta}_0  := \sqrt{(2p)^{2\tild p}|\tilD|^{p-1}}(\log((2p)^{2\tild p}|\tilD|^{p-1}))^{\tild(p-1)}\times \left(\prod\limits_{\substack{v \in \tilS\\  v\nmid \infty}}\log\mathcal{N}_{\tilK/\mathbb{Q}}(v)\right)^{p-1},$$
$$D^*:=p^{d^2\frac{(p-1)^3}{8}}|D|^{\frac{p-1}{2}},$$
\begin{equation*}
\begin{aligned}
\Delta(p) := \sqrt{(2p)^{dp(p-1)}|D^*|^{p-1}}(\log((2p)^{dp(p-1)}|D^*|^{p-1}))^{d\frac{(p-1)^2}{2}} \times \left(\prod\limits_{\substack{v \in S \\ v  \nmid \infty}}\log\mathcal{N}_{K/\mathbb{Q}}(v)\right)^{\frac{(p-1)^2}{2}},
\end{aligned}
\end{equation*}
where $\tild:=[\tilK:\mathbb{Q}]$, and $\tilD$ is the absolute discriminant of $\tilK$.

We follow the idea of \cite{sha2014bounding}. Let $\tilde{s}= |\tilS|$, then $\tilde{s}\leq [\tilK:K]s\leq \dfrac{p-1}{2}s$ and $\tild \leq d\dfrac{p-1}{2}$. For the absolute discriminant $\tilD$ of $\tilK$, we have
\begin{equation*}
\begin{aligned}
|\tilD| &= |\mathcal{N}_{K/\mathbb{Q}}(D_{\tilK/K})||D|^{[\tilK:K]}\\
&\leq p^{d^2\frac{(p-1)^3}{8}}|D|^{\frac{p-1}{2}}\\
&= D^*.
\end{aligned}
\end{equation*}

Now let $w$ be a non-Archimedean place of $K$, and $v_1,\dots, v_m$ be all its extensions to $\tilK$ with residue degrees $f_1,\dots, f_m$ respectively over $K$. Then $f_1+ \dots + f_m \leq [\tilK:K] \leq \dfrac{p-1}{2}$, which implies $\log_2 f_1+ \dots +\log_2 f_m \leq f_1+ \dots + f_m \leq \dfrac{p-1}{2}$, i.e. $f_1\dots f_m \leq 2^{\frac{p-1}{2}}$. Since $\mathcal{N}_{\tilK/\Q}(v_k) = \mathcal{N}_{K/\mathbb{Q}}(w)^{f_k}$ for $1 \leq k \leq m$, we have 
$$\prod\limits_{v | w}\log\mathcal{N}_{\tilK/\mathbb{Q}}(v) \leq 2^{\frac{p-1}{2}}(\log\mathcal{N}_{K/\mathbb{Q}}(w))^{\frac{p-1}{2}}.$$
Hence 
$$\prod\limits_{\substack{v \in \tilS\\  v\nmid \infty}}\log\mathcal{N}_{\tilK/\mathbb{Q}}(v) \leq 2^{s\frac{p-1}{2}}\left(\prod\limits_{\substack{v \in S\\  v\nmid \infty}}\log\mathcal{N}_{K/\mathbb{Q}}(v)\right)^{\frac{p-1}{2}},$$
and 
\begin{equation*}
\begin{aligned}
\tilde{\Delta}_0  &= \sqrt{(2p)^{2\tild p}|\tilD|^{p-1}}(\log((2p)^{2\tild p}|\tilD|^{p-1}))^{\tild(p-1)}\times \left(\prod\limits_{\substack{v \in \tilS\\  v\nmid \infty}}\log\mathcal{N}_{\tilK/\mathbb{Q}}(v)\right)^{p-1}\\
&\leq \sqrt{(2p)^{dp(p-1)}|D^*|^{p-1}}(\log((2p)^{dp(p-1)}|D^*|^{p-1}))^{d\frac{(p-1)^2}{2}}\times 2^{s\frac{(p-1)^2}{2}}\\ &\ \  \times \left(\prod\limits_{\substack{v \in S\\  v\nmid \infty}}\log\mathcal{N}_{K/\mathbb{Q}}(v)\right)^{\frac{(p-1)^2}{2}}\\
&= 2^{s\frac{(p-1)^2}{2}}\Delta(p).
\end{aligned}
\end{equation*}
By Theorem \ref{bound}, we have
\begin{equation*}
\begin{aligned}
\mbox{h}(j(P))  &= \height(j(Q))\\
&\leq (C\tild
\tils(2p)^2)^{4\tils p}(\log(2\tild p))^{6\tils p}\ell^{2\tild p}\tilde{\Delta}_0\\
&\leq 2^{s\frac{(p-1)^2}{2}}(Cds(p-1)^2p^2)^{2sp(p-1)}(\log(dp(p-1)))^{3sp(p-1)}\ell^{dp(p-1)}\Delta(p)\\
\end{aligned}
\end{equation*}
where $\ell$ is the maximal prime such that there exists $v\in S$ with $v|\ell$.

This bound can be made clearer. Indeed, we have the inequalities
$$D^* \leq e^{d^2p^3/8\log p}|D|^{p/2},$$
\begin{equation*}
\begin{aligned}
\Delta(p) & \leq e^{d^2p^4\log p}(2^d|D|)^{p^2} \cdot (d^2p^4\log p+p^2\log|D|)^{dp^2/2} \times \left(\prod\limits_{\substack{v \in S \\ v  \nmid \infty}}\log\mathcal{N}_{K/\mathbb{Q}}(v)\right)^{p^2}\\
 & \leq e^{4s^2p^4\log p}(2^d|D|)^{p^2} \cdot (d^2p^5/2\log(|D|+1))^{dp^2} \times \left(\prod\limits_{\substack{v \in S \\ v  \nmid \infty}}\log\mathcal{N}_{K/\mathbb{Q}}(v)\right)^{p^2}\\
&\leq e^{7s^2p^4\log p}\left((d^2\log(|D|+1))^d |D|\prod\limits_{\substack{v \in S \\ v  \nmid \infty}}\log\mathcal{N}_{K/\mathbb{Q}}(v)\right)^{p^2}\\
&= e^{7s^2p^4\log p} C_1(K,S)^{p^2},
\end{aligned}
\end{equation*}
and
\begin{equation*}
\begin{aligned}
\mbox{h}(j(P))  &\leq 2^{sp^2}(Cdsp^4)^{2sp^2}(\log d + 2\log p)^{3sp^2}\ell^{dp^2}e^{7s^2p^4\log p}C_1(K,S)^{p^2}\\
& \leq e^{9s^2p^4\log p}\cdot 2^{sp^2}(Cds)^{2sp^2}(2d)^{3sp^2}\ell^{dp^2}C_1(K,S)^{p^2}\\
& \leq e^{9s^2p^4\log p}\left(2^{s} \cdot C^{2s}d^{9s}s^{2s}\ell^d|D|(\log{(|D|+1)})^{d}\prod\limits_{\substack{v \in S \\ v  \nmid \infty}}\log\mathcal{N}_{K/\mathbb{Q}}(v)\right)^{p^2}\\
&= e^{9s^2p^4\log p}C(K,S)^{p^2}.
\end{aligned}
\end{equation*}
Hence we get Theorem \ref{main} if we take $C= 2^{15}$, see \cite{cai2019explicit}.

\section*{Acknowledgements}
	The research of the author is supported by the China Scholarship Council. The author also thanks his supervisors Yuri Bilu and Qing Liu for helpful discussions and valuable suggestions.

\address
\end{document}